\documentclass[12pt]{amsart}

\newtheorem{theorem}{Theorem}[section]
\newtheorem{lemma}[theorem]{Lemma}
\newtheorem{proposition}[theorem]{Proposition}
\newtheorem{corollary}[theorem]{Corollary}
\newtheorem{question}[theorem]{Question}

\theoremstyle{definition}

\newtheorem{remark}[theorem]{Remark}

\newenvironment{proofof}[1]{\noindent{\it Proof of
#1.}}{\hfill$\square$\\\mbox{}}

\def\Z{{\mathbb{Z}}}
\def\F{{\mathbb{F}}}
\def\coef{{\mathrm{Coef}}}
\def\tr{{\mathrm{tr}}}
\usepackage{amscd,amssymb}
\begin{document}

\title[Characteristic free description of semi-invariants of $2\times 2$ matrices]
{Characteristic free description of semi-invariants of $2\times 2$ matrices}

\author[M. Domokos]
{M. Domokos}
\address{Alfr\'ed R\'enyi Institute of Mathematics,
Re\'altanoda utca 13-15, 1053 Budapest, Hungary}
\email{domokos.matyas@renyi.hu}

\thanks{This research was partially supported by National Research, Development and Innovation Office,  NKFIH K 119934.}


%\keywords{nil algebra, nilpotent algebra, matrix invariant, degree bound}

\begin{abstract} A minimal homogeneous generating system of the algebra of semi-invariants of tuples of two-by-two matrices over an infinite field of characteristic two or over the ring of integers is given. 
In an alternative interpretation this yields a minimal system of homogeneous generators 
for the vector invariants of the special orthogonal group of degree four over a field of characteristic two or over the ring of integers. 
An irredundant separating system of semi-invariants of tuples two-by-two matrices  is also determined, it turns out to be independent of the characteristic. 
\end{abstract}

\maketitle

%%%%%%%%%%%%%%%%%%%%%%%%%%%%%%%

\section{Introduction}\label{sec:intro} 

For positive integers $n,m$ 
denote by $M:=(\F^{n\times n})^m$ the space of $m$-tuples of $n\times n$ matrices over an infinite field $\F$. 
The group $SL_n\times SL_n$ (the direct product of two copies of the special linear group over $\F$) acts on this space  as follows: 
\[(g,h)\cdot (A_1,\dots,A_m):=(gA_1h^{-1},\dots,gA_mh^{-1})\] 
for $g,h\in SL_n$, $A_1,\dots,A_m\in \F^{n\times n}$. Write $R=\mathcal{O}(M)^{SL_n\times SL_n}$ for the corresponding \emph{algebra of polynomial invariants}; that is, $R$ consists of the polynomial functions on $M$ that are constant along the $SL_n\times SL_n$-orbits. We call $R$ the \emph{algebra of semi-invariants of $m$-tuples of $n\times n$ matrices} 
(and sometimes we shall denote it by $R_{\F}$, when the indication of the base field seems desirable). 
Interest in this algebra was raised recently by connections to algebraic 
complexity theory, see for example \cite{garg-etal}, \cite{hrubes-etal}, \cite{ivanyos_etal}. Another motivation comes from representation theory of quivers: semi-invariants of tuples of matrices are special instances of semi-invariants  associated to representation spaces of quivers, which are used to construct 
moduli spaces parametrizing quiver representations, see \cite{king}. 

$\F$-vector space spanning sets of $R$ were given independently by Derksen and Weyman \cite{derksen-weyman}, Domokos and Zubkov \cite{domokos-zubkov}, and in characteristic zero by Schofield and Van den Bergh  \cite{schofield-vandenbergh} (these papers deal with the more general setup of semi-invariants of quiver representations). To deduce a finite $\F$-algebra generating set of $R$ one needs an explicit degree bound for the generators. Good degree bounds were given 
recently by Derksen and Makam \cite{derksen-makam:1}, \cite{derksen-makam:2}. 
These results imply that the $\Z$-form $R_{\Z}$ of $R$ (see Section~\ref{sec:Z-form}) is  generated as a ring by an explicitly given finite set of elements (see Theorem~\ref{thm:Z-fingen}), moreover, generators of $R_{\Z}$ yield generators in $R_{\F}$ by base change (see Proposition~\ref{prop:integral-invariants}). 
This raises the question of finding an explicit minimal homogeneous system of generators of the ring $R_{\Z}$ (as an optimal characteristic free description of the generators of $R_{\F}$).  
In the present paper we answer the special case $n=2$ of this question. 

In fact an explicit minimal homogeneous generating system of $R_{\F}$ is known in very few  cases only. 
In the special case $n=2$ and $\F=\mathbb{C}$, a minimal homogeneous generating system of $R_{\mathbb{C}}$ was determined by Domokos \cite{domokos}. 
The generators have degree  $2$ and $4$. (As a byproduct of our calculations here we get a new proof of this statement, see Theorem~\ref{thm:odd-char}.) 
It was pointed out by Domokos and Drensky \cite{domokos-drensky} that the result remains valid for any infinite field $\F$ with 
$\mathrm{char}(\F)\neq 2$ thanks to the characteristic free description of vector invariants of the special orthogonal group due to De Concini and Procesi \cite{deconcini-procesi}.  
Apart from that a minimal homogeneous generating system of $R_{\F}$ is known only when $n=m=3$ by \cite[Theorem 4.3]{domokos:3x3} for $\mathrm{char}(\F)=0$ (see also  \cite[Proposition 11.49]{mukai})  and 
by \cite[Theorem 3]{domokos-drensky:2} for arbitrary infinite $\F$ or for $R_{\Z}$.\footnote{After publication of the present work came to my attention that this account is not complete, see Section~\ref{sec:addendum} at the end of the paper.} 

In the present paper we determine a minimal homogeneous generating system of $R_{\F}$ when $n=2$ and $\mathrm{char}(\F)=2$ (see Theorem~\ref{thm:main}). 
Simultaneously we get in Theorem~\ref{thm:2x2-integral} a minimal system of $\Z$-algebra generators in the $\Z$-form of $\mathcal{O}((\mathbb{Q}^{2\times 2})^m)^{SL_2\times SL_2}$. An irredundant separating system of $SL_2\times SL_2$-invariants on $(\mathbb{F}^{2\times 2})^m$ is provided for an algebraically closed base field of arbitrary characteristic   in Theorem~\ref{thm:separating}. 

Additional interest in the calculations of this paper stems from a connection to the open problem of generating vector invariants of the orthogonal group in characteristic two. 
This is discussed in Section~\ref{sec:orthogonal}, where we deduce that 
in the $4$-dimensional case, the exotic vector invariants of $SO(4,\F)$ (with $\mathrm{char}(\F)=2$) constructed by Domokos and Frenkel \cite{domokos-frenkel:2} together with the well-known  quadratic invariants generate the corresponding algebra of invariants. 

\section{The main result} 

Denote by $x_{ijk}\in \mathcal{O}(M)$ the function mapping 
$(A_1,\dots,A_m)\in M$ to the $(i,j)$-entry of $A_k\in \F^{n\times n}$. 
Then $\mathcal{O}(M)$ is the $mn^2$-variable polynomial algebra  
$\F[x_{ijk}\mid 1\le i,j\le n,\ k=1,\dots,m]$. 
Write $x_k$ for the $n\times n$ matrix over $\mathcal{O}(M)$ whose $(i,j)$-entry is 
$x_{ijk}$ ($k=1,\dots,m$; $1\le i,j\le n$). So $x_1,\dots,x_m$ (called sometimes \emph{generic matrices}) are elements of the noncommutative ring $\mathcal{O}(M)^{n\times n}$.  
The algebra $\mathcal{O}(M)$ is graded in the standard way, namely the generators $x_{ijk}$ have degree $1$. The subalgebra $R$ is generated by homogeneous elements. Write $R_+$ for the subspace of $R$ spanned by its homogeneous elements of positive degree (clearly $R_+$ is an ideal in $R$ and 
$R/R_+\cong\F$). By the graded Nakayama lemma some homogeneous elements in $R_+$ form a minimal homogeneous $\F$-algebra generating system of $R$ if and only if they constitute an $\F$-vector space basis in an $\F$-vector space direct complement of the ideal 
$(R_+)^2$ in $R_+$. A homogeneous semi-invariant $f\in R_+$ is said to be \emph{indecomposable} if $f$ is not contained in $(R_+)^2$ (equivalently, $f$ can not be expressed as a polynomial of semi-invariants of strictly smaller degree). 

Now suppose that $n=2$, and let us construct some semi-invariants. 
Clearly the determinant $\det(x_k)$ is a semi-invariant for $k=1,\dots,m$, and we also have its linearization; that is,  
for $1\le k,l\le m$ define $\langle x_k\vert x_l\rangle$ by the equality 
\[\det(z x_k+w x_l)=z^2\det(x_k)+zw \langle x_k\vert x_l\rangle
+w^2\det(x_l)\in \mathcal{O}(M)[z,w] \] 
(where $z,w$  are additional  commuting variables). 
That is, 
\[\langle x_k\vert x_l\rangle=x_{11k}x_{22l}+x_{11l}x_{22k}-x_{12k}x_{21l}-x_{12l}x_{21k}.\] 
Fix an integer  $q$ with $4\le 2q\le m$,  take $2q$ additional commuting variables $z_1,\dots,z_q,w_1,\dots,w_q$, and define 
$\xi_q=\xi_q(x_1,\dots,x_{2q})$ as the coefficient of $z_1\cdots z_qw_1\cdots w_q$ in the determinant 
\begin{equation}\label{eq:xi-def}\left|\begin{array}{ccccc}z_1x_1 & w_1x_{q+1} &  &  &  \\ & z_2x_2 & w_2x_{q+2} &   &  \\  &  & z_{3}x_3 & \ddots &  \\  &  &  & \ddots & w_{q-1}x_{2q-1} \\ w_qx_{2q} &  &   &  & z_qx_q\end{array}\right|\in\mathcal{O}((\F^{2\times 2})^m)[z_1,\dots,z_q,w_1,\dots,w_q]\end{equation} 
where we put the $2\times 2$ zero matrix in those positions of the above $q\times q$ 
block matrix that are not indicated (so in \eqref{eq:xi-def} we take the determinant of a $2q\times 2q$ matrix over $\mathcal{O}(M)[z_1,\dots,z_q,w_1,\dots,w_q]$). 
If $y_1,\dots,y_{2q}$ are $n\times n$ matrices with entries in an arbitrary commutative 
$\F$-algebra $A$, we shall write $\xi_q(y_1,\dots,y_{2q})$ for the image of $\xi_q(x_1,\dots x_{2q})$ under the unique $\F$-algebra homomorphism 
from $\F[x_{ijk}\mid 1\le i,j\le n,\ k=1,\dots,2q]\to A$ that maps $x_{ijk}$ to the $(i,j)$-entry of $y_k$. Now we are in position to state the main results of the paper:  

\begin{theorem}\label{thm:main} 
Let $\F$ be an infinite field of characteristic $2$. Then the following set is a minimal homogeneous generating system of the algebra $\mathcal{O}((\F^{2\times 2})^m)^{SL_2\times SL_2}$:  
\begin{align}\label{eq:2x2-generators} \{\det(x_k),\ \langle x_l\vert x_r\rangle ,\ \xi(x_{k_1},\dots,x_{k_{2q}}) 
\mid 
k=1,\dots,m; \ 1\le l<r\le m;\ 
\\ \notag 2\le q\le \frac m2; \ 1\le k_1<\cdots <k_{2q}\le m\}.
\end{align}
\end{theorem} 

\begin{theorem} \label{thm:2x2-integral}
The ring $\Z[x_{ijk}\mid 1\le i,j\le 2,\ k=1,\dots,m]\cap \mathcal{O}((\mathbb{Q}^{2\times 2})^m)^{SL_2\times SL_2}$ is minimally generated by the elements \eqref{eq:2x2-generators} (defined over $\F=\mathbb{Q}$). 
\end{theorem}

\section{A spanning set of semi-invariants} 

We need to develop some notation in order to recall a known $\F$-vector space spanning set of $R$. Given an $n\times n$ matrix $a=(a_{ij})_{i,j=1}^n$ and a $q\times q$ matrix 
$b=(b_{rs})_{r,s=1}^q$ with entries in some commutative ring $S$ we write $a\otimes b\in S^{nq\times nq}$ for their \emph{Kronecker product}. That is, $a\otimes b$ is 
 a $q\times q$ block matrix with blocks of size $n\times n$, and the block in the $(r,s)$ position is $b_{rs}a\in S^{n\times n}$. 

In the following discussion $n$ is arbitrary, and recall that $M=(\F^{n\times n})^m$. 
Fix another positive integer $q$ and set $N:=(\F^{q\times q})^m$. 
Then the coordinate ring of $M\times N$ is 
\[\mathcal{O}(M\times N)=\mathcal{O}(M)\otimes\mathcal{O}(N)=\F[x_{ijk},t_{rsk}\mid 1\le i,j\le n;\ 1\le r,s\le q; \ k=1,\dots,m],\] 
where $t_{rsk}$ is the function mapping an element $(B_1,\dots,B_q)\in N$ to the $(r,s)$-entry of $B_k$.  
For $k=1,\dots ,m$ set 
\[t_k:=(t_{rsk})_{1\le r,s\le q}\in \mathcal{O}(N)^{q\times q},\quad 
\mbox{ and }t:=(t_1,\dots,t_m).\] 
Then $x_k$ and $t_k$ are both matrices over $\mathcal{O}(M\times N)$, 
so their Kronecker product $x_k\otimes t_k$ is an $nq\times nq$ matrix with entries in the ring $\mathcal{O}(M\times N)$. 
We shall write 
\[x\otimes t:=x_1\otimes t_1+\cdots +x_m\otimes t_m \in 
\mathcal{O}(M\times N)^{nq\times nq}.\] 
More generally, for any $m$-tuple $u=(u_1,\dots,u_m)$ of $q\times q$ matrices over $\mathcal{O}(N)$ we shall use the notation 
\[x\otimes u:=x_1\otimes u_1+\cdots+x_m\otimes u_m\in \mathcal{O}(M\times N)^{nq\times nq}.\] 

For $\alpha\in (\mathbb{N}_0^{q\times q})^m$ set 
\[t^{\alpha}:=\prod_{r,s,k}t_{rsk}^{\alpha_{rsk}}\in \mathcal{O}(N)\subset \mathcal{O}(M\times N).\] 
An element $f\in \mathcal{O}(M,N)$ can be viewed as a polynomial over $\mathcal{O}(M)$ with variables $t_{rsk}$, and denote by 
$\coef(t^{\alpha},f)\in \mathcal{O}(M)$ the coefficient of $t^{\alpha}$ in $f$: 
\[f=\sum_{\alpha\in (\mathbb{N}_0^{q\times q})^m} \coef(t^{\alpha},f)t^{\alpha}.\] 

A well known construction of semi-invariants is based on the multiplicativity of the determinant: 

\begin{lemma}\label{lemma:semiinvariant} 
For any $u\in (\mathcal{O}(N)^{q\times q})^m$ and 
$\alpha\in (\mathbb{N}_0^{q\times q})^m$ we have that 
$\coef(t^{\alpha},\det(x\otimes u))$ belongs to $R$. 
\end{lemma} 

\begin{proof} 
For $g\in SL_n$ write $\mathrm{diag}(g,\dots,g)$ for the element of $SL_{qn}$ obtained by 
glueing $q$ copies of $g$ along the diagonal. Given $(g,h)\in SL_n\times SL_n$ we have the matrix equality 
\[\mathrm{diag}(g^{-1},\dots,g^{-1})\cdot x\otimes u \cdot \mathrm{diag}(h,\dots,h)
=\sum_{k=1}^m(g^{-1}x_kh)\otimes u.\] 
Take determinants, and use the multiplicativity of the determinant to conclude that 
\[\det(x\otimes u)=\det(\sum_{k=1}^m(g^{-1}x_kh)\otimes u).\] 
The latter equality shows that the coefficients of $\det(x\otimes u)$ (viewed as a polynomial in the $t_{rsk}$ with coefficients in $\mathcal{O}(M)$) are $SL_n\times SL_n$-invariant. 
\end{proof} 

Our starting point is the following result of Derksen and Weyman \cite{derksen-weyman}, 
Domokos and Zubkov \cite{domokos-zubkov}, and (in characteristic zero) of 
Schofield and Van den Bergh \cite{schofield-vandenbergh}; in positive characteristic this result was obtained using the theory of modules with good filtration (see  
\cite{donkin:2} and the references there): 

\begin{theorem}\label{thm:semi-invariants} 
Let $\F$ be an infinite field. Then the degree of a non-zero homogeneous element in $R_+$ is a multiple of $n$, and for $q=1,2,\dots$, the degree $nq$ homogeneous component of  the algebra $R$ of matrix semi-invariants is spanned as an $\F$-vector space by 
\begin{equation}\label{eq:dwdzsv}\{\coef(t^{\alpha},\det(x \otimes t))\mid \alpha\in (\mathbb{N}_0^{n\times n})^m \ \mbox{ with }\ 
\sum_{r,s,k}\alpha_{rsk}=nq\}.
\end{equation}  
\end{theorem} 

\section{The main reductions} 

In this section we turn to the case $n=2$, write $M$ for $(\F^{2\times 2})^m$ and  $R$ for 
$\mathcal{O}(M)^{SL_2\times SL_2}$. 

\begin{lemma}\label{lemma:1} Take $q\ge 2$ and 
$\alpha\in (\mathbb{N}_0^{q\times q})^m$ such that $\coef(t^{\alpha},\det(x\otimes t))$ is indecomposable in $R$. Then the following hold: 
\begin{itemize} 
\item[(i)] For all $r,s\in \{1,\dots,q\}$ we have $\sum_{k=1}^m\alpha_{rsk}\in \{0,1\}$. 
\item[(ii)] The semi-invariant $\coef(t^{\alpha},\det(x\otimes t))$ up to sign agrees with 
$\xi(x_{k_1},\dots,x_{k_{2q}})$, where $k_1,\dots,k_{2q}$ are elements of $\{1,\dots,m\}$. 
\end{itemize}
\end{lemma} 

\begin{proof} (i) 
When we expand the determinant of $x\otimes t$, 
each of the $(2q)!$ terms contains at most two entries from each $2\times 2$ block of $x\otimes t$. Since $\coef(t^{\alpha},\det(x\otimes t))\neq 0$, it follows that for all $r,s\in \{1,\dots,q\}$ we have $\sum_{k=1}^m\alpha_{rsk}\le 2$. 
Suppose contrary to our statement (i) that  $\sum_{k=1}^m\alpha_{rsk}=2$ for some $r,s$. By symmetry we may assume  that $r=s=1$.  Laplace expansion of the determinant of 
$x\otimes t$ along the first two rows shows that 
\[\coef(t^{\alpha},\det(x\otimes t))=\coef(\prod_{k=1}^mt_{11k}^{\alpha_{11k}},
\det(\sum_{k=1}^mt_{11k}x_k))\cdot \coef(\prod_{r,s=2}^q\prod_{k=1}^mt_{rsk}^{\alpha_{rsk}}, \det(x\otimes u)),\] 
where $u\in (\mathcal{O}(N)^{(q-1)\times (q-1)})^m$ is obtained from $t=(t_1,\dots,t_m)$ by removing the first row and column of each component $t_k\in \mathcal{O}(N)^{q\times q}$. 
The two factors on the right hand side above both belong to $R_+$ by Lemma~\ref{lemma:semiinvariant}. 
Since $\coef(t^{\alpha},\det(x\otimes t))$ is indecomposable, this is a contradiction. 
Thus $\sum_{k=1}^m\alpha_{rsk}\le 1$ for all $r,s$. 

(ii) Consider the set of pairs $P:=\{(r,s)\mid \sum_{k=1}^m\alpha_{rsk}=1\}$. 
Each of the $(2q)!$ terms of the expansion of the determinant of the matrix $x\otimes t$ contains exactly two factors from the first two rows (respectively columns), two factors from 
the second two rows (respectively columns), etc. Taking into account (i), we get that for each 
$r=1,\dots,q$ we have 
$|P\cap \{(r,1),(r,2),\dots,(r,q)\}|=2$, and  
 for each 
$s=1,\dots,q$ we have $|P\cap \{(1.s),(2,s),\dots,(q,s)\}|=2$. 
A theorem of K\H onig \cite{konig} (asserting that a regular bipartite graph has a perfect matching) implies that there exist permutations $\sigma,\pi$ of the set $\{1,2,\dots,q\}$ such that $\sigma(r)\neq \pi(r)$ for each $r=1,\dots,q$, and 
$P=\{(r,\sigma(r)),(r,\pi(r))\mid r=1,\dots,q\}$. By symmetry (or simultaneously renumbering the rows of the matrix components of $t$) we may assume that 
$\sigma$ is the identity permutation, and so $\pi$ is fix point free. We claim that $\pi$ is 
in fact a $q$-cycle. Suppose to the contrary that $\pi$ preserves a proper non-empty subset 
of $\{1,\dots,q\}$. Again after a possible simultaneous renumbering of both the rows and columns of the matrix components of $t$ we may assume that $\pi$ preserves 
the subset $\{1,\dots,d\}$ (and then necessarily $\pi$ preserves the complement 
$\{d+1,\dots,q\}$). This means that each component $\alpha_k\in \mathbb{N}_0^{q\times q}$ of $\alpha\in (\mathbb{N}_0^{q\times q})^m$ is block diagonal, so 
$\alpha_k=\mathrm{diag}(\beta_k,\gamma_k)$, where $\beta_k\in\mathbb{N}_0^{d\times d}$ 
and $\gamma_k\in\mathbb{N}_0^{(q-d)\times (q-d)}$. 
Set $u=(u_1,\dots,u_m)$, where $u_k$ is the left upper $d\times d$ block of $t_k$, 
and set $v=(v_1,\dots,v_m)$, where $v_k$ is the complementary block to $u_k$ in $t_k$. 
The Laplace expansion of the determinant of $x\otimes t$ along the first $2d$ rows shows that 
\[\coef(t^{\alpha},\det(x\otimes t))=\coef(t^{\beta},\det(x\otimes u))\cdot 
\coef(t^{\gamma},\det(x\otimes v)),\] 
where $\beta=(\beta_1,\dots,\beta_m)\in (\mathbb{N}_0^{d\times d})^m$ and 
$\gamma=(\gamma_1,\dots,\gamma_m)\in (\mathbb{N}_0^{(q-d)\times (q-d)})^m$. 
Both factors on the right hand side above belong to $R_+$ by Lemma~\ref{lemma:semiinvariant}, so this is a contradiction. This proves our claim that $\pi$ is a $q$-cycle. 
Moreover, after a possible simultaneous renumbering of both the rows and columns of all components of $t$ we may assume that $\pi$ is the $q$-cycle $(1,2,\dots,q)$. 
Thus there exist unique elements $k_1,\dots,k_{2q}\in \{1,\dots,m\}$ such that 
\begin{align*}
\alpha_{rsk}=\begin{cases} 1 &\mbox{ for }(r,s,k)\in \{(r,r,k_r)
\mid r=1,\dots,q\} \\
1 & \mbox{ for }(r,s,k)\in \{(r,r+1,k_{q+r}) \mid r=1,\dots,q-1\} \\
1&  \mbox{ for }(r,s,k)=(q,1,k_{2q})\\
0 &\mbox{ otherwise}.\end{cases}
\end{align*}
For this $\alpha$ we have $\coef(t^{\alpha},\det(x\otimes t))=\xi(x_{k_1},\dots,x_{2q})$. 
Since in the process of reducing to this special $\alpha$ we had to renumber rows and columns of the components of $t$, the original semi-invariant we started with agrees with this latter up to sign. 
\end{proof} 

For $f,h\in R$ write $f\equiv h$ if $f-h\in (R_+)^2$. 
In particular, a homogeneous element of $R$ is not indecomposable if and only if 
$f\equiv 0$. 

\begin{lemma}\label{lemma:2} Let $q\ge 2$ be an integer with $2q\le m$. 
\begin{itemize} 
\item[(i)] For any permutation $\pi$ of the set $\{1,\dots,2q\}$ we have 
\[\xi(x_{\pi(1)},\dots,x_{\pi(2q)})\equiv \mathrm{sign}(\pi) \xi(x_{\pi(1)},\dots,x_{\pi(2q)}).\] 
\item[(ii)] If $k_1,\dots,k_{2q}\in \{1,\dots,m\}$ are not distinct, then 
$\xi(x_{k_1},\dots,x_{k_{2q}})\equiv 0$. 
\end{itemize}
\end{lemma} 

\begin{proof} 
We first investigate what happens if we make the substitution $x_{q+1}\mapsto x_1$ in 
$\xi_q$.  
Using standard column operations on matrices (with entries in the field of fractions of the polynomial ring $\mathcal{O}(M)[z_1,\dots,z_q,w_1,\dots,w_q]$) 
we have the following: 
\[\left|\begin{array}{ccccc}z_1x_1 & w_1x_1 &  &  &  \\ & z_2x_2 & w_2x_{q+2} &   &  \\  &  & z_{3}x_3 & \ddots &  \\  &  &  & \ddots & w_{q-1}x_{2q-1} \\ w_qx_{2q} &  &   &  & z_qx_q\end{array}\right|=
\left|\begin{array}{ccccc}z_1x_1 & 0 &  &  &  \\ & z_2x_2 & w_2x_{q+2} &   &  \\  &  & z_{3}x_3 & \ddots &  \\  &  &  & \ddots & w_{q-1}x_{2q-1} 
\\ w_qx_{2q} &-\frac{w_1w_q}{z_1}x_{2q}  &   &  & z_qx_q\end{array}\right|
\]
By Laplace expansion along the first two rows the $2q\times 2q$ determinant on the right hand side above equals 
\[|z_1x_1|\cdot \left|\begin{array}{cccc}  z_2x_2 & w_2x_{q+2} &   &  \\    & z_{3}x_3 & \ddots &  \\    &  & \ddots & w_{q-1}x_{2q-1} 
\\ -\frac{w_1w_q}{z_1}x_{2q}  &   &  & z_qx_q\end{array}\right| 
=|x_1|\cdot \left|\begin{array}{cccc}  z_2x_2 & w_2x_{q+2} &   &  \\    & z_{3}x_3 & \ddots &  \\    &  & \ddots & w_{q-1}x_{2q-1} 
\\ -w_1w_qx_{2q}  &   &  & z_1z_qx_q\end{array}\right| .
\] 
This implies that substituting $x_{q+1}$ and $x_1$ by the same matrix 
$y\in \mathcal{O}(M)^{n\times n}$ we get 
\[\xi_q(y,x_2,\dots,x_q,y,x_{q+2},\dots,x_{2q})=-\det(y)\cdot \xi_{q-1}(x_2,\dots,x_q,x_{q+2},\dots,x_{2q})\]
for $q>2$ and 
 \[\xi_4(y,x_2,y,x_4)=-\det(y)\cdot \langle x_2|x_4\rangle.\]
 In particular, setting $y=x_1+x_{q+1}$ we conclude that 
 \begin{equation}\label{eq:yy}
 \xi_q(x_1+x_{q+1},x_2,x_3,\dots,x_q,x_1+x_{q+1},x_{q+2},x_{q+3},\dots,x_{2q})\equiv 0.
 \end{equation}  
Now note that  $R$ is an $\mathbb{N}_0^m$-graded 
subalgebra of $\F[x_{ijk}\mid 1\le i,j\le n;\ k=1,\dots,m]$, where the $\mathbb{N}_0^m$-grading is given by setting the degree of the entries of $x_k$ to be the $k^{\mathrm{th}}$ standard basis vector $(0,\dots,0,1,0,\dots,0)$ of $\mathbb{N}_0^m$. \
Therefore any multihomogeneous component of the left hand side of \eqref{eq:yy} must belong to $(R_+)^2$. Observe also that $\xi_q$ is additive in each of its $2q$ 
matrix arguments. So taking the multihomogeneus component of multidegree 
$(\underbrace{1,\dots,1}_{2q},0,\dots,0)$ of \eqref{eq:yy} we get  
\[\xi_q(x_1,x_2,\dots,x_q,x_{q+1},x_{q+2},\dots,x_{2q})\equiv
-\xi_q(x_{q+1},x_2,x_3,\dots,x_q,x_1,x_{q+2},x_{q+3},\dots,x_{2q}).\] 
Similar argument shows that $\xi_q(x_1,\dots,x_{2q})$ changes sign if we interchange 
any two matrix arguments that appear in the same row or in the same column of the 
block matrix in \eqref{eq:xi-def}. Since any permutation of the matrix arguments can be obtained by applying an appropriate sequence of such transpositions, we obtain that 
(i) holds. 

In view of statement (i), it is sufficient to prove (ii) in the special case when 
$k_1=k_{q+1}$. The multihomogeneous component of \eqref{eq:yy} of multidegree 
$(2,\underbrace{1,\dots,1}_{q-1},0,\underbrace{1,\dots,1}_{q-1},0,\dots,0)$ gives 
\[\xi_q(x_1,x_2,\dots,x_q,x_1,x_{q+2},x_{q+3},\dots,x_{2q})\equiv 0,\] 
from which the desired statement follows by making the substitution 
$x_d\mapsto x_{k_d}$, $d=1,\dots,2q$. 
\end{proof}

\begin{proofof}{Theorem~\ref{thm:main}} Combining Theorem~\ref{thm:semi-invariants}, Lemma~\ref{lemma:1}, and Lemma~\ref{lemma:2} we get that the homogeneous semi-invariants in the statement of Theorem~\ref{thm:main} indeed generate the $\F$-algebra $R$. To prove minimality of this generating system, note that 
each of the given generators are multihomogeneous with respect to the $\mathbb{N}_0^m$-grading of $R$ mentioned in the proof of Lemma~\ref{lemma:2}. Therefore it is sufficient to show that each generator is indecomposable. 
This is obvious for $\det(x_k)$ and $\langle x_l\vert x_r\rangle$ since they have minimal possible total degree. Indecomposability of $\xi(x_{k_1},\dots,x_{k_{2q}})$ for 
$1\le k_1<\cdots<k_{2q}\le m$ and $\mathrm{char}(\F)=2$ was proved by 
Domokos and Frenkel \cite[Proposition 2.5]{domokos-frenkel:1}. 
\end{proofof}

We finish this section by pointing out that Theorem~\ref{thm:semi-invariants}, Lemma~\ref{lemma:1} and Lemma~\ref{lemma:2} allow an alternative deduction of the zero characteristic case of the following known minimal generating system of $\mathcal{O}((\mathbb{F}^{2\times 2})^m)^{SL_2\times SL_2}$:  

\begin{theorem}\label{thm:odd-char} 
For an infinite field $\mathbb{F}$ of characteristic different from $2$, 
\begin{align}\label{eq:odd-char} 
\{\det(x_k),\ \langle x_l\vert x_r\rangle ,\ \xi(x_{k_1},x_{k_2},x_{k_3},x_{k_4}) 
\mid 
k=1,\dots,m; \ 1\le l<r\le m;\ 
\\  \notag  1\le k_1<\cdots <k_4\le m\}.
\end{align}
is a minimal homogeneous generating system of 
$\mathcal{O}((\mathbb{F}^{2\times 2})^m)^{SL_2\times SL_2}$.  
\end{theorem}  

\begin{proof} As explained in the introduction, this is known by \cite{deconcini-procesi}, \cite{domokos}, \cite{domokos-drensky}. We give a new argument for the case 
$\mathrm{char}(\mathbb{F})=0$. 

Combining Theorem~\ref{thm:semi-invariants}, Lemma~\ref{lemma:1}, and Lemma~\ref{lemma:2} we get that the semi-invariants \eqref{eq:2x2-generators} generate the $\F$-algebra $R$. Therefore to prove that \eqref{eq:odd-char} is already a generating system of $R$, it is sufficient to show that $\xi(x_1,\dots,x_{2q})$ is decomposable in $R$ for each $q>2$. 

Suppose to the contrary that $\xi(x_1,\dots,x_{2q})\notin (R_+)^2$. 
We may assume that $m=2q$. Denote by $T$ the subspace of $R$ spanned by the elements having multidegree 
$(\underbrace{1,\dots,1}_{2q})$. The symmetric group $\Sigma_{2q}=\mathrm{Sym}\{1,\dots,2q\}$ acts linearly on $R$ via $\mathbb{F}$-algebra automorphisms, namely a permutation $\pi$ 
sends $x_{ijk}$ to $x_{ij\pi^{-1}(k)}$. Obviously both $T$ and $R_+^2$ are $\Sigma_{2q}$-invariant subspaces in $R$. By Lemma~\ref{lemma:2} (i) we have that the image of 
$\xi(x_1,\dots,x_{2q})$ in $T/((R_+)^2\cap T)$ spans a one-dimensional $\Sigma_{2q}$-invariant subspace on which $\Sigma_{2q}$ acts via its sign representation. 
In particular, the $\Sigma_{2q}$-module $T$ has the sign representation as a direct summand. This is a contradiction. Indeed, 
$e:=\frac{1}{(2q)!}\sum_{\pi\in\Sigma_{2q}}\mathrm{sign}(\pi)\pi$ is the primitive central idempotent in the group algebra $\mathbb{F}\Sigma_{2q}$ corresponding to the sign representation, 
and $e\cdot T=\{0\}$ when $2q>4$.  

To see minimality of this  homogeneous generating system, note that the given generators are multihomogeneous with distinct multidegrees, so it is sufficient to show that each of them is indecomposable. The quadratic generators have minimal possible degree in $R_+$, therefore they are indecomposable. The degree four generators are not contained 
in the subalgebra of $R$ generated by the quadratic generators. 
Indeed, suppose to the contrary that $\xi(x_1,x_2,x_3,x_4)$ is contained in the subalgebra 
generated by the quadratic generators. Since the quadratic generators take the same value 
on $(A_1,A_2,A_3,A_4)$ and the transposed $4$-tuple $(A_1^T,A_2^T,A_3^T,A_4^T)$, 
the same holds then for $\xi(x_1,x_2,x_3,x_4)$.  However, this is not the case, as one can easily check. 
\end{proof} 

\section{$\mathbb{Z}$-algebra generators} \label{sec:Z-form}

The algebra $\mathcal{O}((\mathbb{Q}^{n\times n})^m)$ contains the subring 
$\mathcal{O}_{\Z}:=\Z[x_{ijk}\mid 1\le i,j\le n,\ k=1,\dots,m]$, and 
$R_{\mathbb{Q}}:=\mathcal{O}((\mathbb{Q}^{n\times n})^m)^{SL_n\times SL_n}$ contains the subring 
\[R_{\Z}:=\mathcal{O}_{\Z}\cap R_{\mathbb{Q}}.\] 
For an infinite field $\F$ identify 
$\mathcal{O}_{\F}:=\mathcal{O}((\F^{n\times n})^m)$ with $\F\otimes \mathcal{O}_{\Z}$  in the obvious way (here and later the symbol $\otimes$ stands for tensor product of $\Z$-modules), and denote by 
\[\iota_{\F}:\mathcal{O}_{\Z}\to \mathcal{O}_{\F}\text{ the map }h\mapsto 1\otimes h.\] 
Since $R_{\Z}$ can be interpreted as a ring of invariants of the  affine group scheme 
$SL_n\times SL_n$ over $\Z$, we have 
\[\iota_{\F}(R_{\Z})\subseteq R_{\F}\] 
(see the book of Jantzen \cite[I.2.10]{jantzen} for the relevant material on invariants of group schemes).  Even more, we have the following: 

\begin{proposition}\label{prop:integral-invariants} 
Let $\F$ be an infinite field.
\begin{itemize}  
\item[(i)] The algebra $R_{\F}$ is spanned as an $\F$-vector space by $\iota_{\F}(R_{\Z})$. 
\item[(ii)] Any system of generators of the ring $R_{\Z}$ is mapped to an $\F$-algebra 
generating system of $R_{\F}$. 
\end{itemize} 
\end{proposition}  

\begin{proof} 
For $\F=\mathbb{Q}$ the elements in \eqref{eq:dwdzsv} belong to $R_{\Z}$, and Theorem~\ref{thm:semi-invariants} asserts that for an arbitrary infinite field $\F$ their images span $R_{\F}$ as an $\F$-vector space. So (i) holds, and it implies (ii).  
\end{proof} 

The polynomial rings $\mathcal{O}_{\Z}$, $\mathcal{O}_{\F}$ are graded in the standard way (the variables $x_{ijk}$ have degree $1$), and the subrings $R_{\Z}$, $R_{\F}$ are generated by homogeneous elements. Below for a graded ring $A$ we whall write $A^{(d)}$ for the degree $d$ homogeneous component of $A$. 

\begin{lemma} \label{lemma:SR} 
Suppose that $S$ is a graded $\Z$-submodule of $R_{\Z}$ such that $\iota_{\F}(S)$ spans 
$R_{\F}$ for any infinite field $\F$. Then $S=R_{\Z}$. 
\end{lemma} 

\begin{proof} 
Note that the factor $\Z$-module $\mathcal{O}_{\Z}/R_{\Z}$ is torsion-free. Indeed, 
if $kh\in R_{\Z}$ for some positive integer $k$ and $h\in \mathcal{O}_{\Z}$, then 
$h\in \mathcal{O}_{\Z}\cap \mathbb{Q}R_{\Z}
=\mathcal{O}_{\Z}\cap R_{\mathbb{Q}}=R_{\Z}$. It follows that for any non-negative integer 
$d$, $\mathcal{O}_{\Z}^{(d)}/R_{\Z}^{(d)}$ is a finitely generated free $\Z$-module. 
Therefore $R_{\Z}^{(d)}$ is a $\Z$-module direct summand in $\mathcal{O}_{\Z}^{(d)}$, hence the embedding $R_{\Z}\hookrightarrow \mathcal{O}_{\Z}$ induces an isomorphism 
\begin{equation}\label{eq:iso}\F\otimes R_{\Z}\stackrel{\cong} \longrightarrow \F\iota_{\F}(R_{\Z})\subseteq \F\otimes \mathcal{O}_{\Z}=\mathcal{O}_{\F}.
\end{equation} 
Now tensor with $\F$ the short exact sequence 
\[0\to S^{(d)}\to R_{\Z}^{(d)}\to R_{\Z}^{(d)}/S^{(d)}\to 0\] 
to get the exact sequence 
\begin{equation}\label{eq:exact} 
\F\otimes S^{(d)}\stackrel{\varphi}  \longrightarrow \F\otimes R_{'\Z}^{(d)} 
\to \F\otimes R_{\Z}^{(d)}/S^{(d)}\to 0.
\end{equation}
The image of $\varphi$ above is $\F\iota_{\F}(S^{(d)})=R_{\F}^{(d)}$ by assumption, which by 
Proposition~\ref{prop:integral-invariants} (i) and by \eqref{eq:iso} is 
$R_{\F}^{(d)}=\F\iota_{\F}(R_{\Z}^{(d)})=\F\otimes R_{\Z}^{(d)}$. 
So $\varphi$ is surjective, and thus the exactness of \eqref{eq:exact} means that $\F\otimes R_{\Z}^{(d)}/S^{(d)}=0$.  
This holds for any infinite field $\F$. 

Now $R_{\Z}^{(d)}/S^{(d)}$ is a finitely generated $\Z$-module. The equality 
$\mathbb{Q}\otimes R_{\Z}^{(d)}/S^{(d)}=0$ implies that $R_{\Z}^{(d)}/S^{(d)}$ is a torsion 
$\Z$-module, hence it is a finite abelian group. We claim that it is zero. Assume on the contrary that it is non-zero, and let $p$ be a prime divisor of its order. Then for any field $\F$ whose characteristic is $p$, we have 
$\F\otimes R_{\Z}^{(d)}/S^{(d)}\neq 0$, a contradiction. 
So we have $R_{\Z}^{(d)}/S^{(d)}=0$ holds for all $d\in \mathbb{N}_0$. This means that 
$S=R_{\Z}$. 
\end{proof} 

\begin{remark} A similar argument was used by Donkin \cite[Page 399]{donkin} 
to get generators in the $\Z$-form of the algebra of simultaneous conjugation invariants of $m$-tuples of $n\times n$ matrices. 
\end{remark} 

As a special case of a general result of Seshadri \cite{seshadri} one gets that  
the ring $R_{\Z}$ is finitely generated. An explicit finite generating system is given below: 

\begin{theorem}\label{thm:Z-fingen}   
The ring $R_{\Z}$ is generated by the elements in \eqref{eq:dwdzsv}  with $q=1,2,\dots,mn^3$, where $\F=\mathbb{Q}$. 
\end{theorem}  

\begin{proof} 
Denote by $S$ the subring of $R_{\Z}$ generated by the elements in \eqref{eq:dwdzsv}  with $q=1,2,\dots,mn^3$, where $\F=\mathbb{Q}$. For any infinite field $\F$, 
$\F \iota_{\F}(S)$ is a subalgebra of $R_{\F}$. By Theorem~\ref{thm:semi-invariants} it 
contains the homogeneous components of $R_{\F}$ of degree at most $mn^4$. 
Derksen and Makam \cite{derksen-makam:2} proved that these components generate the $\F$-algebra $R_{\F}$. Thus we have $\F\iota_{\F}(S)=R_{\F}$  for any infinite field $\F$. 
Consequently, by Lemma~\ref{lemma:SR} we conclude $S=R_{\Z}$. 
\end{proof}

\begin{remark}
Similar arguments together with the results of \cite{derksen-makam:2}, \cite{derksen-weyman}, \cite{domokos-zubkov}  yield finite generating systems of the $\Z$-forms of rings of semi-invariants on spaces of quiver representations. 
We omit the details. 
\end{remark} 

Next we return to the special case $n=2$: 

\medskip
\begin{proofof}{Theorem~\ref{thm:2x2-integral}} 
Denote by $S$ the subring of $\Z[x_{ijk}\mid 1\le i,j\le 2,\ k=1,\dots,m]$ generated by the elements \eqref{eq:2x2-generators}.  The conditions of Lemma~\ref{lemma:SR} hold for 
$S$ by Theorem~\ref{thm:semi-invariants}, Lemma~\ref{lemma:1}  and Lemma~\ref{lemma:2},  hence by Lemma~\ref{lemma:SR} we get that $S=R_{\Z}$. 

Minimality of this generating system follows from the minimality statement in Theorem~\ref{thm:main} (which follows from \cite{domokos-frenkel:1}) and Proposition~\ref{prop:integral-invariants} (ii). 
\end{proofof} 

\section{Minimal system of separating invariants} 

In this section we assume that the base field $\F$ is algebraically closed. 
Following Derksen and Kemper \cite{derksen-kemper:book},  
a subset $S$ of $R=\mathcal{O}(M)^{SL_n\times SL_n}$ is called \emph{separating} if for any $A,B\in M$, if there exists an 
$f\in R$ with $f(A)\neq f(B)$, then there exists an $h\in S$ with $h(A)\neq h(B)$. 

\begin{theorem} \label{thm:separating} Let $\mathbb{F}$ be an algebraically closed field of arbitrary characteristic. 
\begin{itemize} 
\item[(i)] The following is a separating system of 
$\mathcal{O}((\F^{2\times 2})^m)^{SL_2\times SL_2}$: 
\begin{align}\notag 
\mathcal{S}_m:=\{\det(x_k),\ \langle x_l\vert x_r\rangle ,\ \xi(x_{k_1},x_{k_2},x_{k_3},x_{k_4}) 
\mid 
k=1,\dots,m; \ 1\le l<r\le m;\ 
\\  \notag  1\le k_1<\cdots <k_4\le m\}.
\end{align}
\item[(ii)] The above separating system is irredundant. 
\end{itemize}
\end{theorem} 

\begin{proofof}{Theorem~\ref{thm:separating} (i)}  
In the special case 
$m=4$, $\mathcal{S}_4$ is even a generating system by Theorem~\ref{thm:main} 
and Theorem~\ref{thm:odd-char}. 
Since $\dim_{\F}(\F^{2\times 2})=4$,  
a general result of Draisma, Kemper and Wehlau 
\cite{draisma-kemper-wehlau} or Grosshans \cite[Theorem 7]{grosshans} implies that the polarizations of a separating system in $\mathcal{O}((\F^{2\times 2})^4)$ constitute a separating system 
in $\mathcal{O}((\F^{2\times 2})^m)$ for an arbitrary $m$. Now for $m\ge 4$ the  polarizations of the subset $\mathcal{S}_4$ are contained in the subalgebra generated by  
$\mathcal{S}_m$.  
\end{proofof} 

In order to prove Theorem~\ref{thm:separating} (ii) it is helpful to recall a connection between the ring of matrix semi-invariants and another ring of invariants. Namely the group $SL_n$ 
acts on $(\mathbb{F}^{n\times n})^m$ by simultaneous conjugation: for $g\in SL_n$ and 
$(A_1,\dots,A_m)$ we set 
\[g\cdot (A_1,\dots,A_m)=(gA_1g^{-1},\dots,gA_mg^{-1}).\] 
For $m\ge 2$ consider the embedding 
\[\sigma:\mathcal{O}((\mathbb{F}^{n\times n})^{m-1})\to 
\mathcal{O}((\mathbb{F}^{n\times n})^m),\quad (A_1,\dots,A_{m-1})\mapsto (A_1,\dots,A_{m-1},I)\] 
where $I$ stands for the $n\times n$ identity matrix. 

\begin{proposition}\label{prop:4.1}  \cite[Proposition 4.1]{domokos:3x3} 
The comorphism 
$\sigma^*:\mathcal{O}((\mathbb{F}^{n\times n})^m)\to \mathcal{O}((\mathbb{F}^{n\times n})^{m-1})$ of $\sigma$ maps $\mathcal{O}(M)^{SL_n\times SL_n}$ surjectively onto 
the algebra $\mathcal{O}((\mathbb{F}^{n\times n})^{m-1})^{SL_n}$ of simultaneous conjugation invariants of matrices.   
\end{proposition} 

This statement has the following immediate corollary (cf. \cite[Proposition 3.2]{derksen-makam:orbitclosure}):  

\begin{corollary}\label{cor:sep-to-sep} 
Any separating system of $SL_n\times SL_n$-invariants on $(\mathbb{F}^{n\times n})^m$ is mapped by $\sigma^*$ to a separating system of simultaneous conjugation $SL_n$-invariants on $(\mathbb{F}^{n\times n})^{m-1}$. 
\end{corollary} 

\begin{proof} Let $S\subset \mathcal{O}(M)^{SL_n\times SL_n}$ be a separating system. 
 Suppose that  for $A=(A_1,\dots,A_{m-1})\in (\mathbb{\F}^{n\times n})^{m-1}$ and 
$B=(B_1,\dots,B_{m-1})\in (\mathbb{\F}^{n\times n})^{m-1}$ there exists an 
$f\in \mathcal{O}((\mathbb{\F}^{n\times n})^{m-1})^{SL_n}$ with $f(A)\neq f(B)$. 
By Proposition~\ref{prop:4.1} there exists a $b\in \mathcal{O}(M)^{SL_n\times SL_n}$ with 
$\sigma^*(b)=f$. Then we have $b(\sigma(A))\neq b(\sigma(B))$, therefore there exists an 
$h\in S$ with $h(\sigma(A))\neq h(\sigma(B))$. Now $\sigma^*(h)\in \sigma^*(S)$ separates $A$ and $B$. 
\end{proof} 

For sake of completeness we recall that for a representation of a reductive group, 
two points can be separated by polynomial invariants if and only if the unique Zariski closed orbits in their orbit closures are different, see for example \cite[Theorem A]{grosshans:book}. 
Therefore Corollary~\ref{cor:sep-to-sep} follows from the following geometric refinement 
of Proposition~\ref{prop:4.1}. 
The group $SL_n$ acts faithfully on $SL_n\times SL_n$ via $a\cdot (g,h):=(ga^{-1},ha^{-1})$, hence we can form the asociated fibre bundle $(SL_n\times SL_n)\times_{SL_n}(\mathbb{F}^{n\times n})^{m-1}$  (see  for example \cite[Lemma 5.16]{bongartz} for the properties of associated fibre bundles needed below).

\begin{proposition} \label{prop:fibre-bundle} (special case of  \cite[Lemma 3.1]{domokos-zubkov}) 
Set 
\[M_0:=\{(A_1,\dots,A_m)\in (\mathbb{F}^{n\times n})^m\mid \det(A_m)=1\}.\]  
The map 
\[SL_n\times SL_n\times (\mathbb{F}^{n\times n})^{m-1}\longrightarrow M_0, 
\quad (g,h,A_1,\dots,A_{m-1})\mapsto (gA_1h^{-1},\dots,gA_{m-1}h^{-1},gh^{-1})\] 
is $SL_n$-invariant,  and factors through an isomorphism 
\[(SL_n\times SL_n)\times_{SL_n}(\mathbb{F}^{n\times n})^{m-1}\stackrel{\cong}\to M_0.\] 
In particular, the map $\sigma$ induces a bijection between $SL_n$-invariant subvarieties in 
$(\mathbb{F}^{n\times n})^{m-1}$ and $SL_n\times SL_n$-invariant subvarieties in $M_0$. 
This bijection respects inclusions and Zariski closures, so restricts to  a bijection between Zariski closed $SL_n$-orbits in $(\mathbb{F}^{n\times n})^{m-1}$ and 
Zariski closed $SL_n\times SL_n$-orbits in $M_0$. 
\end{proposition} 

\begin{remark}
In a recent preprint  Derksen and Makam \cite{derksen-makam:orbitclosure} gave an algorithm that decides whether the $SL_n\times SL_n$-orbit closures of $A,B\in M$ intersect. 
\end{remark}

We need the following observation of I. Kaygorodov, A. Lopatin and Y. Popov 
\cite[Lemma 3.1]{kaygorodov-lopatin-popov}: 

\begin{lemma}\label{lemma:kaygorodov-etal} The following is an irredundant separating system of conjugation $SL_2$-invariants on 
$(\mathbb{F}^{2\times 2})^3$: 
\[\tr(x_k), \ \det(x_k), \ k=1,2,3; \ \tr(x_lx_r),\ 1\le l<r\le 3; \ \tr(x_1x_2x_3).\]
\end{lemma} 

Note that we have 
\begin{equation}\label{eq:tr(xy)}
\langle x_l|x_r\rangle=\tr(x_l)\tr(x_r)-\tr(x_lx_r).
\end{equation} 

\medskip

\begin{proofof}{Theorem~\ref{thm:separating} (ii)} 
Assume to the contrary that omitting an element $h$ from $\mathcal{S}_m$ we still have a separating system. 
Now $h$ depends only on at most $4$ matrix components. By symmetry we may assume 
that $h$ does not depend $x_5,\dots,x_m$, so $h\in \mathcal{S}_4$. 
Identify $(\F^{2\times 2})^4$ with the subspace of $(\F^{2\times 2})^m$ consisting of the tuples whose last $m-4$ matrix components are zero. Since all elements in $\mathcal{S}_m\setminus \mathcal{S}_4$ vanish identically on the subspace $(\F^{2\times 2})^4$ in 
$(\F^{2\times 2})^m$, we conclude that $\mathcal{S}_4\setminus \{h\}$ is a separating system in $\mathcal{O}((\F^{2\times 2})^4)$. 

Suppose first that $h=\xi(x_1,x_2,x_3,x_4)$. Denote by $\sigma:(\mathbb{F}^{2\times 2})^3\mapsto (\mathbb{F}^{2\times 2})^4$ the embedding $(A_1,A_2,A_3)\mapsto (A_1,A_2,A_3,I)$ (where $I$ is the $2\times 2$ identity matrix). 
Then \eqref{eq:tr(xy)} shows that  $\sigma^*(\mathcal{S}_4\setminus \{h\})$ is contained in the subalgebra of $\mathcal{O}((\mathbb{F}^{2\times 2})^3)$ generated by the elements of degree at most two in the separating system given in Lemma~\ref{lemma:kaygorodov-etal}. Therefore by Lemma~\ref{lemma:kaygorodov-etal}, $\sigma^*(\mathcal{S}_4\setminus \{h\})$ is not a separating system on $(\mathbb{F}^{2\times 2})^3$ with respect to the conjugation action of $SL_2$, contrary to Corollary~\ref{cor:sep-to-sep}. 

If $h$ does not depend on all the four matrix components, say $h$ does not depend on $x_4$, 
then we get that a proper subset of $\{\det(x_1),\ \det(x_2),\ \langle x_1|x_2\rangle\}$ is a separating system on $(\mathbb{F}^{2\times 2})^2$ with respect to the  
action of $SL_2\times SL_2$. This is obviously a contradiction. 
\end{proofof} 

The main result of \cite{kaygorodov-lopatin-popov} can be obtained as a consequence 
of Theorem~\ref{thm:separating} (i) and Lemma~\ref{lemma:kaygorodov-etal}: 

\begin{corollary}\label{cor:kaygorodov-etal} (Kaygorodov, Lopatin and Popov 
\cite[Theorem 1.1]{kaygorodov-lopatin-popov}) For an algebraically closed field $\mathbb{F}$ of arbitrary characteristic and a positive integer $m$, the set 
\[\{\tr(x_k), \ \det(x_k),\ \tr(x_lx_r),\ \tr(x_{k_1}x_{k_2}x_{k_3})\mid  \ k=1,2,3; \ 1\le l<r\le m; 
\ 1\le k_1<k_2<k_3\le m\}\]
is an irredundant separating system of $SL_2$-invariants on $(\mathbb{F}^{2\times 2})^m$. 
\end{corollary} 

\begin{proof} We have the equality 
\begin{equation}\label{eq:tr(xyz)} 
\xi(x_1,x_2,x_3,I)=\tr(x_1x_2x_3)-\tr(x_1x_2)\tr(x_3).
\end{equation}  
Indeed, consider \eqref{eq:xi-def} in the special case $q=2$, and substitute 
$x_4$ by the $2\times 2$ identity matrix $I$: 
\[\left|\begin{array}{cc}z_1x_1 & w_1x_3 \\ w_2I & z_2x_2 \end{array}\right|=
\left|\begin{array}{cc}0 & w_1x_3-z_1w_2^{-1}z_2x_1x_2 \\ w_2I & z_2x_2 \end{array}\right|
=|w_1w_2x_3-z_1z_2x_1x_2|.\] 
Now \eqref{eq:tr(xy)} gives \eqref{eq:tr(xyz)}, which implies by   
Theorem~\ref{thm:separating} (i) and Corollary~\ref{cor:sep-to-sep} 
(applied for $m+1$ instead of $m$) that the set in our statement is indeed a separating system. Its irredundancy is an immediate consequence of Lemma~\ref{lemma:kaygorodov-etal}. 
\end{proof}

\begin{remark} We note that similarly to the proof of Theorem~\ref{thm:Z-fingen}, using 
Lemma~\ref{lemma:SR} one gets form the results of \cite{donkin} and \cite{derksen-makam:2} that the $\Z$-form $\Z[x_{ijk}\mid 1\le i,j\le n,\ k=1,\dots,m]\cap \mathcal{O}((\mathbb{Q}^{n\times n})^m)^{SL_n}$ of the $\mathbb{Q}$-algebra of simultaneous conjugation invariants of matrices is generated by coefficients of the characteristic polynomials of monomials in the generic matrices $x_1,\dots,x_m$, that have total degree at most $(m+1)n^4$. 
\end{remark}

\section{Vector invariants of the special orthogonal group}\label{sec:orthogonal}

Up to base change there is only one non-singular quadratic form on $\mathbb{F}^4=\mathbb{F}^{2\times 2}$, namely the form $A\mapsto \det(A)$ ($A\in \mathbb{F}^{2\times 2}$). The \emph{orthogonal group} $O(4,\mathbb{F})$ therefore can be defined as the subgroup of 
$GL(\mathbb{F}^{2\times 2})$ consisting of the linear transformations that preserve the determinant. The connected component of the identity in $O(4,\mathbb{F})$ is an index two subgroup, called the \emph{special orthogonal group} $SO(4,\mathbb{F})$, and it is well known that 
$SO(4,\mathbb{F})=\rho(SL_2\times SL_2)$, where $\rho:SL_2\times SL_2\to GL(\mathbb{F}^{2\times 2})$ is the representation of $SL_2\times SL_2$ on $\mathbb{F}^{2\times 2}$ given by $(g,h)\cdot A:=gAh^{-1}$. 
So the representation of $SL_2\times SL_2$ on $(\mathbb{F}^{2\times 2})^m$ can be thought of as the sum of $m$ copies of the defining representation of $SO(4,\mathbb{F})$, 
and the algebra $R=\mathcal{O}(M)^{SL_2\times SL_2}=\mathcal{O}((\mathbb{F}^{2\times 2})^m)^{SO(4,\mathbb{F})}$ is the so-called \emph{algebra of vector invariants of the special orthogonal group of degree} $4$. 

Generators of the algebra  $\mathcal{O}((\mathbb{F}^n)^m)^{SO(n,\mathbb{F})}$ have a uniform description in characteristic zero, and this is one of the theorems usually referred to as the First Fundamental Theorem of Invariant Theory, see \cite{weyl}. 
The result extends verbatim (but with essentially different proof) for odd positive characteristic by \cite{deconcini-procesi}. However, in the case 
$\mathrm{char}(\mathbb{F})=2$,  
for all even $m$ an indecomposable multilinear $SO(n,\mathbb{F})$-invariant on $(\mathbb{F}^n)^m$ was constructed in \cite[Theorem 7]{domokos-frenkel:2}. 
This led Procesi \cite[page 551]{procesi} to write that "in this case the invariant theory has to be deeply modified". 
In particular, the following question is  open: 

\begin{question} \label{quest:df}
Do the invariants constructed in \cite{domokos-frenkel:2} together with the quadratic invariants generate $\mathcal{O}((\mathbb{F}^n)^m)^{SO(n,\mathbb{F})}$ when $\mathrm{char}(\mathbb{F})=2$? 
\end{question} 

Now as an immediate consequence of Theorem~\ref{thm:main} we get that 
the answer to Question~\ref{quest:df} is affirmative in the special  case $n=4$: 

\begin{corollary}\label{cor:orth-4}
Let $\mathbb{F}$ be an algebraically closed field of characteristic $2$. 
The elements given in \cite[Theorem 7]{domokos-frenkel:2} together with the usual quadratic invariants constitute a minimal generating system of the algebra of vector invariants 
of the special orthogonal group $SO(4,\mathbb{F})$ on $(\mathbb{F}^4)^m$.   
\end{corollary} 

Moreover, formally the same elements as in Corollary~\ref{cor:orth-4} generate the appropriate  $\Z$-form of $\mathcal{O}((\mathbb{C}^4)^m)^{SO(4,\mathbb{C})}$ by 
an argument similar to the proof of Theorem~\ref{thm:2x2-integral}. 
 
\section{Note added after publication}\label{sec:addendum} 

Lopatin in \cite[Theorems 2.3 and 2.20]{lopatin} determines a 
minimal system of generators of the algebra of semi-invariants of an arbitrary quiver with a dimension vector taking value $2$ at each vertex of the quiver, over an arbitrary infinite base  field. In particular, in \cite[Lemma 8.6]{lopatin} Lopatin gives an explicit  minimal generating system of semi-invariants of quivers having two vertices with dimension vector taking value $2$ at both vertices. So as a special case of Lopatin's result, a minimal generating system of the algebra of semi-invariants of $2\times 2$ matrices  was known prior to our 
Theorem~\ref{thm:main} also in characteristic $2$. 

Theorem~\ref{thm:main} is obtained by different calculations than the corresponding result in \cite{lopatin}, and the other topics of the present paper (separating invariants and the connection to the orthogonal group) are not discussed in \cite{lopatin}.

%%%%%%%%%%%%%%%%%%%%%%%%%%%

\end{document}